\documentclass[letterpaper,11pt]{article}
\usepackage{amsfonts,amsmath,amssymb,amsthm}

\newtheorem{thm}{Theorem}[section]
\newtheorem{prop}[thm]{Proposition}
\newtheorem{lem}[thm]{Lemma}
\newtheorem{cor}[thm]{Corollary}
\newtheorem{con}[thm]{Conjecture}
\newtheorem{asm}{Assumption}

\theoremstyle{remark}
\newtheorem{rem}[thm]{Remark}

\theoremstyle{definition}

\newcommand{\ra}{\rightarrow}

\newcommand{\la}{\leftarrow}

\newcommand{\N}{\mathbb N}     % For Natural numbers
\newcommand{\Z}{\mathbb Z}     % For Integers

\renewcommand{\a}{\alpha}
\renewcommand{\b}{\beta}

\renewcommand{\d}{\delta}
\newcommand{\e}{\varepsilon}
\renewcommand{\k}{\kappa}

\renewcommand{\l}{\lambda}
\renewcommand{\L}{\Lambda}
\newcommand{\s}{\sigma}

\newcommand{\fl}[1]{\lfloor #1 \rfloor}           % Floor function
\newcommand{\cl}[1]{\lceil #1 \rceil}		  % Ceiling function
\newcommand{\indd}[1]{ \mathbf{1}{ \{ #1 \} } }   % Indicator function that doesn't make a subscript 

\newcommand{\be}{\begin{equation}}
\newcommand{\ee}{\end{equation}}
\newcommand{\nn}{\nonumber}

\newcommand{\iid}{i.i.d.\ }                       % Usually makes i.i.d. look better in text
\newcommand{\w}{\omega}                     % Shortcut for \omega
\renewcommand{\P}{\mathbb{P}}               % Annealed Probability
\newcommand{\E}{\mathbb{E}}                 % Annealed Expectation
\newcommand{\vp}{\mathrm{v}_P}              % Limiting velocity
\newcommand{\tw}{\tilde{\omega}}	    % Tilted environment. 

\title{Maximal Displacement for Bridges of Random Walks in a Random Environment}
\author{Jonathon Peterson \and Nina Gantert}
\date{October 23, 2009}

\begin{document}

\maketitle

\begin{abstract}
It is well known that the distribution of simple random walks on $\Z$ conditioned on returning to the origin after $2n$ steps does not depend on $p= P(S_1 = 1)$, the probability of moving to the right. Moreover, conditioned on $\{S_{2n}=0\}$ the maximal displacement $\max_{k\leq 2n} |S_k|$ converges in distribution when scaled by $\sqrt{n}$ (diffusive scaling). 

We consider the analogous problem for transient random walks in random environments on $\Z$. We show that under the quenched law $P_\w$ (conditioned on the environment $\w$), the maximal displacement of the random walk when conditioned to return to the origin at time $2n$ is no longer necessarily of the order $\sqrt{n}$. If the environment is nestling (both positive and negative local drifts exist) then the maximal displacement conditioned on returning to the origin at time $2n$ is of order $n^{\kappa/(\kappa+1)}$, where the constant $\kappa>0$ depends on the law on environment. 
On the other hand, if the environment is marginally nestling or non-nestling (only non-negative local drifts) then the maximal displacement conditioned on returning to the origin at time $2n$ is at least $n^{1-\e}$ and at most $n/(\ln n)^{2-\e}$ for any $\e>0$. 

As a consequence of our proofs, we obtain precise rates of decay for $P_\w(X_{2n}=0)$. In particular, for certain non-nestling environments we show that $P_\w(X_{2n}=0) = \exp\{ -Cn -C'n/(\ln n)^2 + o(n/(\ln n)^2) \}$ with explicit constants $C,C'>0$. 
\end{abstract}

\section{Introduction.}
Let $(S_n)_{ n\geq 1}$ be a random walk with drift on the integers: let $p \in (0,1)$ and let $Y_1, Y_2, \ldots $ be a sequence of \iid random 
variables with $P(Y_1 =+1) = p = 1- P(Y_1 = -1)$, and $S_n = \sum\limits_{i=1}^n Y_i, n=1,2,3, \ldots $. Consider the conditioned law of $(\frac{S_k}{\sqrt{2n}})_{1\leq k \leq 2n}$, conditioned on the event $\{S_{2n} =0\}$. It is easy to see that this conditioned laws converge to the law of a Brownian bridge for $n \to \infty$, for any value of $p \in (0,1)$. In fact, for $p=\frac{1}{2}$, this is a consequence of Donsker's invariance principle, but by symmetry, the conditioned laws do not depend on the value of $p$. In particular,
the laws 
\be\label{srw}
P\left(\frac{1}{\sqrt{2n}} \max\limits_{1 \leq k \leq 2n}|S_k| \in \cdot \; \Bigl| S_{2n =0} \right)
\ee
converge, for $ n \to \infty$, to the law of the maximum of the absolute value of a Brownian bridge.

In this paper, we address a question originally posed to us by Itai Benjamini.
The question is, what happens if we replace $(S_n)$ with a random walk in random environment 
$(X_n)_{ n\geq 1}$?
In particular, what is the order of growth of $\max\limits_{1 \leq k \leq 2n}|X_k|$, conditioned on the event $\{X_{2n} =0\}$?
Is it diffusive as in (\ref{srw}) or superdiffusive or subdiffusive, respectively?

The random walk in random environment (abbreviated RWRE) is defined as follows.
Let $\omega =(\omega_x)_{x \in \Z}$ be a collection of random variables taking values in 
$(0,1)$ and let $P$ be the distribution of $\omega$. The random variable $\omega$ is called the ``random environment''. For each environment $\omega \in \Omega= (0,1)^{\Z}$ and $y\in \Z$, we define the 
RWRE starting at $y$ as the time-homogeneous Markov chain $(X_n)$ taking values in $\Z$, with $X_0=y$ and transition probabilities 
\[
P_\omega^y [X_{n+1} = x+1 \, |\, X_n =x] = \omega_x = 1-P_\omega^x [X_{n+1} = x-1 \, | \, X_n =x] , \quad n\geq 0.
\]
We equip $\Omega$ with its Borel $\sigma$-field
$\cal{F}$ and $\Z^\N$ with its Borel $\sigma$-field $\cal{G}$. The distribution of $(\omega, (X_n))$ is the probability measure $\P$ on $\Omega \times \Z^\N$ defined by 
\[
 \P[F\times G] = \int\limits_F P_\omega[G] P(d\omega), \quad F \in \mathcal{F}, \quad G \in\cal{G}.
\]
Since we will usually be concerned with random walks starting from the origin we will use $P_\w$ and $E_\w$ to denote $P_\w^0$ and $E_\w^0$, respectively. 
Expectations with respect to $P$, $P_\w$, and $\P$ will be denoted by $E_P$, $E_\w$, and $\E$, respectively.
$P_\w$ is referred to as the \emph{quenched} law of the random walk, while $\P$ is referred to as the averaged (or annealed) law.

The goal of this paper is to study the magnitude of the maximal displacement $\max\limits_{1 \leq k \leq 2n}|X_k|$ of a RWRE under the quenched law $P_\w$,  conditioned on the event $\{X_{2n} =0\}$. 
For a simple random walk, (\ref{srw}) tells us that the scaling is of order $\sqrt{n}$ (diffusive scaling). 
This is not the case for RWRE as we will show below.
In fact, super-diffusive, sub-diffusive and diffusive scaling are possible (depending on the law $P$ of the environment). 

Throughout the paper we will make the following assumptions. 
\begin{asm}\label{UEIIDasm}
The environment is i.i.d.\ and uniformly elliptic. That is, the random variables $\{\w_x\}_{x\in\Z}$ are \iid under the measure $P$, and there exists a constant $c > 0$ such that $P(\w_x \in [c, 1-c] ) = 1$. 
\end{asm}
Let $\rho_i= \rho_i(\omega): = (1-\omega_i)/\omega_i$, $i \in \Z$.
\begin{asm} \label{Tasm}
 $E_P \log \rho_0 < 0$. 
%That is, the RWRE is transient to $+\infty$. 
\end{asm}
As shown in \cite{sRWRE}, the second assumption implies that for $P$-almost all $\omega$, the Markov chain $(X_n)$ is transient and we have $X_n \to +\infty$, $P_\w$-a.s.
A lot more is known about this one-dimensional model; we will not give background here, but we refer to the survey 
by Ofer Zeitouni~\cite{zRWRE} for 
limit theorems, large deviations results, and for further references. 

Our main results are as follows.
%Let $\w_{\min}:= \inf \{ x: \, x\in \hbox{\rm supp}(\omega_0)\}$. 
Let $\w_{\min}:= \inf \{ x: P(\w_0 \leq x) > 0 \}$. 
We will distinguish between three different cases: $\w_{\min} < \frac{1}{2}$ (nestling case), $\w_{\min} = \frac{1}{2}$ (marginally nestling case) and $\w_{\min} > \frac{1}{2}$ (non-nestling case). It turns out that in the nestling case, the magnitude of $\max\limits_{1 \leq k \leq 2n}|X_k|$ under the conditioned law is of order $n^\beta$, with an exponent $\beta \in (0,1)$ depending on the law of the environment (see Theorem \ref{nestlingcase} for the precise statement). The cases $\beta > \frac{1}{2}$, $\beta = \frac{1}{2}$ and $\beta < \frac{1}{2}$ are all possible, and we have $\beta > \frac{1}{2}$ if and only if $(X_n)$ has a positive linear speed. In the marginally nestling and the non-nestling cases, we give additional assumptions excluding the deterministic case (i.e. the case of constant environment). 
We then show that the magnitude of $\max\limits_{1 \leq k \leq 2n}|X_k|$ under the conditioned law is between $n^{1-\e}$ and $n/(\ln n)^{2-\e}$ for any $\e>0$ (see Theorems \ref{margnestlingcase} and \ref{nestlingcase} for precise statements). We conjecture in fact that the correct order of magnitude is $n/(\ln n)^2$ (see Conjecture \ref{probconv}). 
%We then expect the magnitude of $\max\limits_{1 \leq k \leq 2n}|X_k|$ under the conditioned law to be of order $\frac{n} {(\log n)^2}$, see Theorem \ref{margnestlingcase}, Theorem \ref{nestlingcase} and Conjectures \ref{truly2} and \ref{probconv} for precise statements. In particular, we show that in the marginally nestling and non-nestling (but not deterministic) cases, the scaling is always superdiffusive.

As a consequence of the proofs of our main theorems we also obtain precise asymptotics on the rate of decay of the quenched probabilities $P_\w(X_{2n} = 0)$. In the nestling and marginally nestling cases the decay rates are $\exp\{-n^{\k/(\k+1)+o(1)}\}$ and $\exp\{-Cn/(\ln n)^2+o(n/(\ln n)^2)\}$ for an explicit constant $C>0$, respectively (see Lemmas \ref{Xneq0} and \ref{mnatzero}). These decay rates are not surprising given the previous results on moderate and large deviations for RWRE in \cite{fgpMD} and \cite{ppQSubexp}. The non-nestling case turns out to be more interesting. 
%It was known previously from large deviation results that $P_\w(X_{2n}=0)$ decays exponentially fast if the environment is non-nestling. 
It was known previously from large deviation results that $P_\w(X_{2n}=0) = \exp\{-Cn +o(n)\}$ for an explicit constant $C>0$. 
For our results, however, we needed some sub-exponential corrections to this rate of decay. In Corollary \ref{nnatzero} we show in fact that 
\[
 P_\w(X_{2n}= 0) = \exp \left\{ -Cn -C'n/(\ln n)^2 + o(n/(\ln n)^2) \right\}, \qquad P-a.s.,
\]
for explicit positive constants $C$ and $C'$ depending on the law on environments. 

The paper is organized as follows. In Sections \ref{nestlingsection}, \ref{margnestlingsection} and \ref{nonnestlingsection} respectively, the nestling, marginally nestling and non-nestling cases are treated. In Section \ref{openproblemsection} we state a conjecture on an improved lower bound for the maximal displacement in the marginally nestling and non-nestling cases.

We conclude this section with some notation that will be used throughout the rest of the paper. 
We will use $\theta$ to denote the standard shift operation on doubly infinite sequences. That is, $(\theta^x \w)_y = \w_{x+y}$. 
Also, we will use $T_k := \inf\{ n\geq 0: X_n = k \}$ to be the hitting time of the site $k$ by the random walk.  The law of a simple symmetric random walk (i.e., $\w_x \equiv 1/2$) will be denoted by $P_{1/2}$ with corresponding expectations denoted $E_{1/2}$. 

{\bf Acknowledgement} We thank the organizers of the program ``Discrete probability'' (spring 2009) for inviting us to beautiful Institut Mittag-Leffler where this paper was initiated. We also thank Itai Benjamini for originally posing this question to us and Ofer Zeitouni for helpful discussions. 

\section{Case I: Nestling environment}\label{nestlingsection}
 
%What does \k do?
Throughout this section we will make the following assumption on environments in addition to Assumptions \ref{UEIIDasm} and \ref{Tasm}. 
\begin{asm}[Nestling] \label{Nasm}
 $P(\w_0 < 1/2) > 0$ and $P(\w_0 > 1/2) > 0$. 
\end{asm}
If Assumptions \ref{UEIIDasm}, \ref{Tasm}, and \ref{Nasm} are satisfied, then we can define a parameter 
$\k = \k(P)$ by 
\be\label{kdef}
 \k = \k(P) \text{ is the unique positive solution of } E_P \rho_0^\k = 1. 
\ee
The parameter $\kappa$ first appeared in \cite{kksStable} in relation to the scaling exponent for the averaged limiting distributions of transient one-dimensional RWRE.
Moreover, the law of large numbers for transient RWRE derived by Solomon in \cite{sRWRE} may be re-stated in terms of the parameter $\kappa$. 
\be\label{LLN}
\lim_{n\ra\infty} \frac{X_n}{n} = 
\begin{cases}
 \frac{ 1- E_P \rho_0}{1+ E_P \rho_0} =: \vp  > 0 & \k > 1 \\
 0 & \k \leq 1.
\end{cases}
\ee
If $\k < 1$, then the typical displacement of $X_n$ is sub-linear and of the order $n^\k$ (see \cite{kksStable} or \cite{eszStable}). 

The main result of this section is that the the maximal displacement of bridges for RWRE under the above assumptions is approximately of the order $n^{\k/(\k+1)}$. 
\begin{thm}\label{nestlingcase}
  Let Assumptions \ref{UEIIDasm}, \ref{Tasm}, and \ref{Nasm} hold. Then, for $P-a.e.$ environment $\w$,
\[
 \lim_{n\ra\infty} P_\w\left( \max_{k\leq 2 n} |X_k| \geq n^\b \; \Bigl| \; X_{2n} = 0 \right) 
= \begin{cases}
   1 & \beta < \frac{\k}{\k+1} \\
   0 & \beta > \frac{\k}{\k+1},
  \end{cases}
\]
where $\k>0$ is defined by \eqref{kdef}.
\end{thm}
\begin{rem}
The cases $\frac{\k}{\k+1} > \frac{1}{2}$, $\frac{\k}{\k+1} =\frac{1}{2}$ and $\frac{\k}{\k+1} < \frac{1}{2}$ are possible. Note that, due to (\ref{LLN}), $\frac{\k}{\k+1} > \frac{1}{2}$ if and only if $(X_n)$ has a.s. positive linear speed.
\end{rem}
An important ingredient in the proof of Theorem \ref{nestlingcase} is the following lemma.
\begin{lem} \label{confinement}
 Let Assumptions \ref{UEIIDasm}, \ref{Tasm}, and \ref{Nasm} hold. If $\beta \in (0,1)$, then for $P-a.e.$ environment $\w$,
\[
 P_\w\left( \max_{k\leq n} |X_k| < n^\b \right) = \exp\left\{ -n^{1-\frac{\b}{\k} + o(1)} \right\}.
\]
\end{lem}
\begin{proof}
For any environment $\w$, let $\w^-$ and $\w^+$ be the modified environment by adding a reflection to the left or right, respectively, at the origin. That is,
\be\label{defrefl}
 (\w^-)_x := \begin{cases} \w_x & x \neq 0 \\ 0 & x = 0, \end{cases} \quad\text{and}\quad 
(\w^+)_x := \begin{cases} \w_x & x \neq 0 \\ 1 & x = 0. \end{cases}
\ee
The random walk starting at the origin stays in the interval $(-n^\b,n^\b)$ if both the left excursions and the right excursions from the origin take at least $n$ steps to leave the interval. Therefore,
\begin{equation}\label{confinelb}
P_\w\left( \max_{k\leq n} |X_k| < n^{\b} \right) 
\geq P_{\w^-} \left( T_{-\cl{n^\b}} > n \right) P_{\w^+} \left( T_{\cl{n^\b}} > n \right).
\end{equation}
%Now since moving the starting point and the reflection to the left only makes exits to the left more likely, 
Explicit formulas for hitting probabilities (see equation (2.1.4) in \cite{zRWRE}) imply that
\[
 P_{\w^-} \left( T_{-\cl{n^\b}} > n \right) 
\geq P_\w^{-1}( T_{-\cl{n^\b}} > T_0 )^{n} 
\geq \left( 1 - \prod_{-\cl{n^\b} < i < 0} \rho_i \right)^{n}. 
\]
The law of large numbers implies that $\prod_{-x<i <0} \rho_i = \exp\{ x( E_P(\log \rho_0) + o(1)) \}$ as $x \ra \infty$, $P-a.s.$ Since $E_P \log \rho_0 < 0$, this implies that the first probability on the right hand side of \eqref{confinelb} tends to $1$ as $n\ra\infty$, $P-a.s.$ In Theorem 1.2 in \cite{fgpMD}, it was shown that for any $\b<\k$
\begin{equation}\label{slowdownreflection}
 P_{\w^+} \left( T_{\cl{n^\b}} > n \right) = \exp\left\{ -n^{ 1- \frac{\b}{\k} + o(1) } \right\}, \quad P-a.s.
\end{equation}
Recalling \eqref{confinelb}, this completes the proof of the lower bound. 

To prove the corresponding upper bound, note that $P_\w\left( \max_{k\leq n} |X_k| < n^\b \right)$ depends only on the $\w_x$ with $|x|< n^{\b}$. Therefore, the probability is unchanged by modifying the environment so that $\w_{-\cl{n^\b}} = 1$. By the strong Markov property, the probability of staying confined to the interval $(-n^\b, n^\b)$ when starting at the origin in the original environment is less than the probability of taking more than $n$ steps to reach $\cl{n^\b}$ when starting at $-\cl{n^\b}$ in the modified environment. That is,
\begin{equation}\label{confineub}
 P_\w\left( \max_{k\leq n} |X_k| < n^\b \right) \leq P_{( \theta^{-\cl{n^\b}} \w)^+}\!\!\left( T_{2\cl{n^\b}} > n \right). 
\end{equation}
At this point, we would like to again apply \eqref{slowdownreflection} to the probability on the right in \eqref{confineub}. However, the presence of the shifted environment in the quenched probability doesn't allow for a direct application. We claim that the proof of \eqref{slowdownreflection} in \cite{fgpMD} may be easily modified to obtain that for any fixed sequence $x_n$, 
\begin{equation} \label{sequenceslowdown}
 P_{( \theta^{x_n} \w)^+}\left( T_{\cl{n^\b}} > n \right) = \exp\left\{ -n^{ 1- \frac{\b}{\k} + o(1) } \right\}, \quad P-a.s. 
\end{equation}
Indeed, in \cite{fgpMD}, it was shown that there were collections of typical environments $\Omega_n \subset \Omega$ such that $\sum_n P(\Omega_n^c) < \infty$ so that the Borel-Cantelli Lemma implied that $\w \in \Omega_n$ for all $n$ large enough, $P-a.s.$ Then, 
%given the control of environments in $\Omega_n$, 
upper and lower bounds were developed for the probabilities $P_{\w^+} \left( T_{\cl{n^\b}} > n \right)$ which are uniform over $\w\in\Omega_n$. These upper and lower bounds and the fact that $\w\in\Omega_n$ for all $n$ large enough imply \eqref{slowdownreflection}. 
For any fixed sequence $x_n$ the shift invariance of $P$ implies that $P(\w \in \Omega_n) = P(\theta^{x_n}\w \in \Omega_n)$, and thus the Borel-Cantelli Lemma implies that $\theta^{x_n} \w \in \Omega_n$ for all $n$ large enough, $P-a.s.$  Therefore, the statement \eqref{sequenceslowdown} follows from the uniform upper and lower bounds for environments in $\Omega_n$. Applying \eqref{sequenceslowdown} to \eqref{confineub} with the sequence $x_n = - \cl{n^\b}$ gives the needed upper bound. 
%\textbf{CHECK MY UNDERSTANDING OF THE PROOF IN \cite{fgpMD}.}
\end{proof}

As a first step in computing the magnitude of displacement of bridges, we first calculate the asymptotic probability of being at the origin at time $2n$. 
\begin{lem} \label{Xneq0}
Let Assumptions \ref{UEIIDasm}, \ref{Tasm}, and \ref{Nasm} hold. Then, for $P-a.e.$ environment $\w$,
\[
 P_\w(X_{2n} = 0) = \exp\left\{-n^{\frac{\k}{\k + 1} + o(1)}\right\}, 
\]
where $\k>0$ is defined by \eqref{kdef}.
\end{lem}
\begin{proof}
This lemma follows easily from Lemma \ref{confinement} and the moderate deviation asymptotics derived in \cite{fgpMD}. If $\nu \in (0,1 \wedge \kappa)$, then Theorem 1.2 in \cite{fgpMD} implies that
\[
 \lim_{n\ra\infty} \frac{ \ln \left( -\ln P_\w( T_{n^\nu} > n ) \right) }{\ln n}  = \lim_{n\ra\infty} \frac{ \ln \left( -\ln P_\w( X_n < n^\nu ) \right) }{\ln n} = \left( 1- \frac{\nu}{\k} \right) \wedge \frac{\k}{\k +1}, \qquad P-a.s.
\]
Note that $1-\frac{\nu}{\k} < \frac{\k}{\k+1}$ if and only if $\nu > \frac{\k}{\k+1}$. Therefore, for $\nu \leq \frac{\k}{\k + 1}$,
\[
 P_\w(X_{2n} = 0) \leq P_\w(X_{2n} < (2n)^{\nu}) = \exp\left\{-n^{\frac{\k}{\k + 1} + o(1)}\right\}. 
\]
To get a corresponding lower bound, we will exhibit a strategy for obtaining $X_{2n} = 0$. 
Force the random walk to stay in the interval $[-n^{\k/(\k+1)}, n^{\k/(\k+1)}]$ for the first $2n-\fl{n^{\k/(\k+1)}}$ steps, and then choose a deterministic path from the random walks current location to force the random walk to be at the origin at time $2n$. Uniform ellipticity (Assumption \ref{UEIIDasm}) and Lemma \ref{confinement} imply
\[
 P_\w(X_{2n} = 0) \geq P_\w\left( \max_{k\leq 2n} |X_k| \leq n^{\k/(\k+1)} \right) c^{2 \fl{n^{\k/(\k+1)}}} = \exp\left\{ -n^{-\frac{\k}{\k+1} + o(1)}  \right\}. 
\]
%Let $K(n)$ be the location of the largest ``trap'' in $[0,n^{\k/(\k+1)}]$. This trap has depth of order $\frac{1}{\k+1} \ln n$ and so the time to escape this trap is roughly exponential with mean $n^{1/(\k+1)}$. Therefore, the probability to stay confined in the trap for $n$ steps is approximately $e^{-n^{1-1/(\k+1)}} = e^{-n^{\k/(\k+1)}}$. Reaching this trap quickly costs almost nothing (probability $\approx 1$) and we can use uniform ellipticity to modify the last $n^{\k/(\k+1)}$ in order to make the random walk be back at the origin at time $n$. Forcing the last $n^{\k/(\k+1)}$ steps costs us approximately $e^{-n^{\k/(\k+1)}}$. 
\end{proof}

We are now ready to give the proof of the main result in this section. 
\begin{proof}[Proof of Theorem \ref{nestlingcase}.]
We first give a lower bound for $P_\w\left( \max_{k\leq n} |X_k| \geq n^\b \;|\; X_{2n} = 0 \right)$. 
\begin{align*}
 P_\w\left( \max_{k\leq 2n} |X_k| \geq n^\b | X_{2n} = 0 \right) 
&= 1 - \frac{ P_\w\left( \max_{k\leq 2n} |X_k| < n^\b ,\; X_{2n} = 0 \right) }{ P_\w( X_{2n} = 0 ) } \\
&\geq 1 - \frac{ P_\w\left( \max_{k\leq 2 n} |X_k| < n^\b \right) }{ P_\w( X_{2n} = 0 ) } \\
&= 1- \frac{ \exp\left\{ -n^{1-\frac{\b}{\k} + o(1)} \right\} }{ \exp\left\{ -n^{\frac{\k}{\k+1} + o(1)} \right\} },
\end{align*}
where the last inequality is from Lemmas \ref{confinement} and \ref{Xneq0}. Since $1-\frac{\b}{\k} > \frac{\k}{\k+1}$ when $\b< \frac{\k}{\k+1}$, we have thus proved Theorem \ref{nestlingcase} in the case $\b< \frac{\k}{\k+1}$. 

Next we turn to the case $\b>\frac{\k}{\k+1}$. First of all, note that 
\[
 P_\w\left( \max_{k\leq 2n} |X_k| \geq n^\b \;|\; X_{2n} = 0 \right) = \frac{ P_\w\left( \max_{k\leq 2n} |X_k| \geq n^\b \;,\; X_{2n} = 0 \right) }{ \exp\left\{ -n^{\frac{\k}{\k+1} + o(1)} \right\} }.
\]
Therefore, it is enough to show that, $P-a.s.$, there exists a $\b'\in (\frac{\k}{\k+1}, \b)$ such that for all $n$ sufficiently large,
\be\label{nestub}
 P_\w\left( \max_{k\leq 2n} |X_k| \geq n^\b \;,\; X_{2n} = 0 \right) \leq e^{-n^{\b'}} .
%= \exp\left\{ -n^{\b'} \right\}
\ee
To this end, note that the event $\left\{ \max_{k\leq 2n} |X_k| \geq n^\b \;,\; X_{2n} = 0 \right\} $ implies that either the random walk reaches $-\fl{n^\b}$ in less than $2n$ steps, or after first reaching $\fl{n^\b}$ the random walk returns to the origin in less than $2n$ steps. Thus, the strong Markov property implies that
\be \label{twobt}
 P_\w\left( \max_{k\leq 2n} |X_k| \geq n^\b \;,\; X_{2n} = 0 \right) \leq P_\w\left( T_{-\cl{n^\b}} < 2 n\right) + P_{\theta^{\cl{n^\b}}\w}\left( T_{-\cl{n^\b}} < 2n \right). 
\ee
It was shown in Theorem 1.4 in \cite{fgpMD} that for any $\b\in(0,1)$,
\be\label{backtracking}
 \lim_{n\ra\infty} \frac{ \ln\left( - \ln \P( T_{-\cl{n^\b}} < 2n ) \right)}{\ln n} = \lim_{n\ra\infty} \frac{ \ln\left( - \ln P_\w( T_{-\cl{n^\b}} < 2n ) \right)}{\ln n} = \b, \quad P-a.s.
\ee
This implies that, $P-a.s.$, the first term on the right hand side of \eqref{twobt} is less than $e^{-n^{\b'}}$ for any $\b'<\b$ and all $n$ large enough. To show that the same is true of the second term on the right hand side of \eqref{twobt}, note that Chebychev's inequality and the shift invariance of $P$ imply that 
\[
 P\left( P_{\theta^{\cl{n^\b}}\w}\left( T_{-\cl{n^\b}} < 2n \right) > e^{-n^{\b'}} \right) \leq e^{n^{\b'}} \P \left( T_{-\cl{n^\b}} < 2n \right) = e^{n^{\b'}} e^{-n^{\b+o(1)}}. 
\]
Since $\b'<\b$, this last term is summable, and thus the Borel-Cantelli Lemma implies that, $P-a.s.$, $P_{\theta^{\cl{n^\b}}\w}\left( T_{-\cl{n^\b}} < 2 n \right) \leq e^{-n^{\b'}}$ for all $n$ large enough. Choosing $\b'\in(\frac{\k}{\k+1}, \b)$ and recalling \eqref{twobt}, we have that, $P-a.s.$, \eqref{nestub} holds for all $n$ large enough. This completes the proof of Theorem \ref{nestlingcase} in the case $\b> \frac{\k}{\k+1}$. 
\end{proof}

\section{Case II: Marginally nestling environment}\label{margnestlingsection}
In this section we will assume the following condition on environments. 
\begin{asm} \label{MNasm}
 $P(\w_0 \geq 1/2 ) = 1$ and $\a=P(\w_0 = 1/2) \in (0,1)$. 
\end{asm}
Note that Assumption \ref{MNasm} implies Assumption \ref{Tasm}, and so in this section we will only assume Assumptions \ref{UEIIDasm} and \ref{MNasm}. 
% $\w_{\min} = \frac{1}{2}$ (i.e., $P$ is marginally nestling) and that $P(\w_0 = \frac{1}{2} ) \in(0,1)$. (Note that this implies that Assumption \ref{Tasm} holds.) 
Our main result in this section is that the displacement of bridges in this case is greater than $n^{1-\e}$ and less than $n/(\ln n)^{2-\e}$ for any $\e>0$. 
\begin{thm}\label{margnestlingcase}
  Let Assumptions \ref{UEIIDasm} and \ref{MNasm} hold. Then
%for $P-a.e.$ environment $\w$,
%\[
% \lim_{n\ra\infty} P_\w\left( \max_{k\leq 2n} |X_k| \geq \frac{n}{(\ln n)^\b} \; \Bigl| \; X_{2n} = 0 \right) 
%= \begin{cases}
%   1 & \b > 2 \\
%   0 & \b < 2.
%  \end{cases}
%\]
for any $\b < 2$,
\begin{equation} \label{mndispub}
 \lim_{n\ra\infty} P_\w\left( \max_{k\leq 2n} |X_k| \geq \frac{n}{(\ln n)^\b} \; \Bigl| \; X_{2n} = 0 \right) = 0, \quad P-a.s.,
\end{equation}
and for any $\gamma < 1$,
\begin{equation}  \label{mndisplb}
  \lim_{n\ra\infty} P_\w\left( \max_{k\leq 2n} |X_k| \geq n^\gamma \; \Bigl| \; X_{2n} = 0 \right) = 1, \quad P-a.s.
\end{equation}
\end{thm}

\begin{rem}
Assumption \ref{MNasm} implies that $\w_{\min} = 1/2$, which is sometimes referred to as the marginally nestling condition. The added condition of $\a \in (0,1)$ was also assumed previously in the quenched and averaged analysis of certain large deviations \cite{gzQSubexp,dpzTE1D,ppzSubexp,ppQSubexp}. We suspect that if $\w_{\min} = 1/2$ but $\a=0$ then there exists a $\gamma < 2$ such that the displacement of bridges is bounded above by $n/(\ln n)^{\gamma - \e}$ for any $\e>0$ (c.f. Remark 1 on page 179 in \cite{gzQSubexp}). 
\end{rem}

The first step in the proof of Theorem \ref{margnestlingcase} is a computation of the precise subexponential rate of decay of the quenched probability to be at the origin. 

%Before giving the proof of Theorem \ref{margnestlingcase} we prove a couple of necessary lemmas. 
\begin{lem}\label{mnatzero}
Let Assumptions \ref{UEIIDasm} and \ref{MNasm} hold. Then,
\[
 \lim_{n\ra\infty} \frac{(\ln n)^2}{n} \ln P_\w(X_{2n} = 0) = - \frac{|\pi \log \a|^2}{4}. 
\]
\end{lem}
\begin{proof}
Theorem 1 in \cite{ppQSubexp} implies that for any $v\in (0,\vp)$
\[
 \limsup_{n\ra\infty} \frac{(\ln n)^2}{n} \log P_\w(X_{2n} = 0) \leq \limsup_{n\ra\infty} \frac{(\ln n)^2}{n} \log P_\w(X_{2n} \leq 2nv ) = - \frac{|\pi \ln \a|^2}{4}\left( 1- \frac{v}{\vp} \right). 
\]
Taking $v\ra 0$ proves the upper bound needed.

For the corresponding lower bound we will force the random walk to stay in a portion of the environment where there is a long sequence of consecutive sites with $\w_x=\frac{1}{2}$. We say $x$ is a \emph{fair} site if $\w_x=\frac{1}{2}$.
For any $n\geq 1$, let 
\[
 L(n) := \max \left\{ j-i: 0\leq i < j \leq n/(\ln n)^3, \; \w_x = \frac{1}{2}, \; \forall i\leq x<j \right\}
\]
be the length of the longest interval of \emph{fair} sites in $[0,n/(\ln n)^3)$, and let
\[
 i_n := \inf\{ i \in [0,n/(\ln n)^3 - L(n)]:\; \w_x = 1/2, \; \forall i\leq x < i+L(n) \}
\]
be the leftmost endpoint of an interval of fair sites in $[0,n/(\ln n)^3]$ of maximal length $L(n)$. 
Theorem 3.2.1 in \cite{dzLDTA} implies that
\begin{equation}\label{maxfair}
 \lim_{n\ra\infty} \frac{L(n)}{\ln n} = \frac{1}{|\ln \a|}, \quad P-a.s. 
\end{equation}
%Let $i_n := \inf\{ i \in [0,n/(\ln n)^3 - L(n)]:\; \w_x = 1/2, \; \forall i\leq x < i+L(n) \}$ be the left endpoint of an interval of fair sites in $[0,n/(\ln n)^3]$ of maximal length $L(n)$. 
Uniform ellipticity (Assumption \ref{UEIIDasm}) implies that 
\begin{align}
 \liminf_{n\ra\infty} \frac{(\ln n)^2}{n} \log P_\w( X_{2n} = 0) &\geq \liminf_{n\ra\infty} \frac{(\ln n)^2}{n} \log P_\w\left( \max_{k\leq 2n} |X_k| < n/(\ln n)^3 \right) \label{uefairconfine} \\
&\geq \liminf_{n\ra\infty} \frac{(\ln n)^2}{n} \log P_\w^{i_n}\left( X_k \in (- n/(\ln n)^3, i_n + L(n)), \; \forall k\leq 2n \right). \nonumber 
\end{align}
%For any environment $\w$, let $\w^-$ and $\w^+$ be the modified environment by adding a reflection to the left or right, respectively, at the origin, recall \ref{defrefl}.
%\[
% (\w^-)_x := \begin{cases} \w_x & x \neq 0 \\ 0 & x = 0, \end{cases} \quad\text{and}\quad 
%(\w^+)_x := \begin{cases} \w_x & x \neq 0 \\ 1 & x = 0. \end{cases},
%\]
The random walk starting at $i_n$ stays in the interval $(-n/(\ln n)^3, i_n+L(n))$ if both the left excursions and right excursions from $i_n$ take more than $2n$ steps to leave the interval. 
Recalling the definitions of $\w^-$ and $\w^+$ given in \eqref{defrefl} we obtain that
%Therefore,
\begin{align}
& P_\w^{i_n}\left( X_k \in (- n/(\ln n)^3, i_n + L(n)), \; \forall k\leq 2n \right) \nonumber \\
&\qquad \geq P_{(\theta^{i_n}\w)^-} \left( T_{-n/(\ln n)^3 - i_n} > 2n \right) P_{(\theta^{i_n}\w)^+} \left( T_{L(n)} > 2n \right). \label{fairtrap}
\end{align}
%Now since moving the starting point and the reflection to the left only makes exits to the left more likely, 

As in the proof of Lemma \ref{mnconfine}, explicit formulas for hitting probabilities and Assumption \ref{Tasm} imply that the first probability on the right hand side of \eqref{fairtrap} tends to 1 as $n\ra\infty$, $P-a.s$.
%Explicit formulas for hitting probabilities (see equation (2.1.4) in \cite{zRWRE}) imply that
%\[
% P_{(\theta^{i_n}\w)^-} \left( T_{-n/(\ln n)^3 - i_n} \geq 2n \right) 
%%\geq P_{\w^-} \left( T_{-n/(\ln n)^3} \geq n \right) 
%\geq P_\w^{-1}( T_{-n/(\ln n)^3} > T_0 )^{2n} 
%\geq \left( 1 - \prod_{-n/(\ln n)^3 < i <  0} \rho_i \right)^{2n}. 
%\]
%The law of large numbers implies that $\prod_{-x<i < 0} \rho_i = \exp\{ x( E_P(\log \rho_0) + o(1)) \}$ as $x \ra \infty$, $P-a.s.$ Since $E_P( \log \rho_0) < 0$, this implies that the first probability on the right hand side of \eqref{fairtrap} tends to $1$ as $n\ra\infty$, $P-a.s.$
%
Due to the fact that $\w_x = 1/2$ for all $i_n\leq x < i_n+L(n)$, the second probability on the right hand side of \eqref{fairtrap} is equal to the probability that a simple symmetric random walk stays in the interval $[-L(n),L(n)]$ for $2n$ steps. 
%Let $P_{1/2}$ be the law of a simple symmmetric random walk (that is $\w_x \equiv 1/2$).
To this end, we recall the following small deviation asymptotics for a simple symmetric random walk (see Theorem 3 in \cite{mSmDev}). 
\begin{lem}\label{SRWsmalldev}
 Let $\lim_{n\ra\infty} x(n) = \infty$ and $x(n) = o(\sqrt{n})$. Then, 
\[
 \lim_{n\ra\infty} \frac{x(n)^2}{n} \ln P_{1/2}\left( \max_{k\leq n} |X_k| \leq x(n) \right) = -\frac{\pi^2}{8}.
\]
\end{lem}
Then, Lemma \ref{SRWsmalldev} and \eqref{maxfair} imply that
\[
 \lim_{n\ra\infty} \frac{(\ln n)^2}{n} \log P_{(\theta^{i_n}\w)^+} \left( T_{L(n)} \geq 2n \right) 
%= \lim_{n\ra\infty} \frac{(\ln n)^2}{n} \log P_{1/2} \left( \max_{k\leq 2n} |X_k| \leq L(n) \right) 
= - \frac{|\pi \ln \a|^2}{4}. 
\]
%\[
% \lim_{n\ra\infty} \frac{(\ln n)^2}{n} \log P_{1/2} \left( X_k \in [-L(n),L(n)], \; \forall k\leq 2 n \right) = - \frac{|\pi \ln \a|^2}{4}. 
%\]
Recalling \eqref{uefairconfine} and \eqref{fairtrap} implies the lower bound needed for the proof of the lemma. 
\end{proof}

To prove Theorem \ref{margnestlingcase} we need to compare the quenched probability to be at the origin at time $2n$ with the quenched probability to stay confined to the interval $[-n^\gamma, n^\gamma]$ for the first $2n$ steps of the random walk. The following proposition says that the exponential rate of the decay of the latter probability is also of the order $n/(\ln n)^2$ but with a larger constant than the probability to be at the origin.

%We also need the following proposition which gives the correct subexponential rate of decay for the quenched probability of confinement in $[-n^\gamma, n^\gamma]$. 
\begin{prop}\label{mnconfine}
Let Assumptions \ref{UEIIDasm} and \ref{MNasm} hold. Then, for any $\gamma \in (0,1)$ 
\[
 \lim_{n\ra\infty} \frac{(\ln n)^2}{n} \ln P_\w \left( \max_{k\leq 2n} |X_k| \leq n^\gamma \right) = - \frac{| \pi \ln \a|^2}{4 \gamma^2 }, \qquad P-a.s. 
\]
\end{prop}
\begin{proof}
As in the proof of Lemma \ref{mnatzero}, a lower bound is obtained by forcing the random walk to stay near a long stretch of fair sites (i.e., sites $x\in\Z$ with $\w_x = 1/2$). However, Theorem 3.2.1 in \cite{dzLDTA} implies that for any $\gamma > 0$, the longest stretch of consecutive fair sites contained in $[0, n^\gamma]$ is $\left( \frac{\gamma}{|\ln \a|} + o(1) \right)\ln n$. The remainder of the proof of the lower bound is the same as in the proof of Lemma \ref{mnatzero} and is therefore omitted. 

The proof of the upper bound is an adaptation of the upper bound for quenched large deviations of slowdowns given by Pisztora and Povel \cite{ppQSubexp}.
It turns out that an upper bound for $P_\w(\max_{k\leq 2n} |X_k| \leq n^\gamma)$ with $\gamma<1$ is substantially easier to prove than the upper bounds for large deviations of the form $P_\w(X_n < nv)$ with $v \in (0,\vp)$. (In fact the multi-scale analysis present in \cite{ppQSubexp} can be avoided entirely.)
%Thus, we will only give a brief outline of the proof of the upper bound in \cite{ppQSubexp} and an explanation of how to modify it to give the necessary upper bound for $P_\w(X_n < n^\gamma)$.

We begin by dividing up space into \emph{fair} and \emph{biased} blocks as was done in \cite{ppQSubexp}. Fix a $\d\in(0,1/3)$ and let $\e,\xi>0$ be such that $0 <\e < P(\w_0 \geq 1/2 + \xi)$ ($\e$ and $\xi$ will eventually be arbitrarily small). Divide $\Z$ into disjoint intervals (which we will call \emph{blocks}) 
\[
 B_j = B_j(n) := \left[ j \fl{(\ln n)^{1-\d}}, (j+1) \fl{(\ln n)^{1-\d}} \right), \quad j\in\Z.
\]
A block $B_j$ is called \emph{biased} if the proportion of sites $x\in B_j$ with $\w_x \geq 1/2 + \xi$ is at least $\e$. 
Otherwise the block is called \emph{fair}. 
Let $J(n,\d,\e,\xi) = J := \{ j\in\Z: B_j \text{ is biased} \}$ be the indices of the biased blocks, and let $J= \cup_{k\in\Z} \{j_k\}$, with $j_k$ increasing in $k$ and $j_{-1} < 0 \leq j_0$. Collect blocks into (overlapping) regions $R_k$ defined by 
\[
 R_k := \bigcup_{j=j_k}^{j_{k+1}} \overline{B}_j, \quad k\in\Z,
\]
where $\overline{B}_j = \left[ j \fl{(\ln n)^{1-\d}}, (j+1) \fl{(\ln n)^{1-\d}} \right]$.
%denotes the closure of $B_j$. 
Each region $R_k$ begins and ends with a biased block and all interior blocks (possibly none) are fair blocks. 

%The proof of the large deviation upper bounds in \cite{ppQSubexp} is accomplished by dividing up space into \emph{fair} and \emph{biased} regions composed of consecutive \emph{blocks} of size $\fl{(\ln n)^{1-\d}}$ for some small $\d>0$. A block is said to be fair if almost all of the sites $x$ in the block have $\w_x$ very close to $1/2$ (the notions of ``almost all sites'' and ``very close to $1/2$'' are made precise in \cite{ppQSubexp}) and otherwise is said to be biased. A fair region is then defined to be a consecutive sequence of blocks of which the first and last block are biased and all of the interior blocks are fair. A biased region is a pair of adjacent biased blocks\footnote{In \cite{ppQSubexp} biased regions are further classified as either \emph{regular} or \emph{irregular}, but this further classification is unnecessary for the proof of Lemma \ref{mnconfine}.}.

Note that with this construction the adjacent regions overlap, but each site $x\in \fl{(\ln n)^{1-\d}} \Z$ (which are the endpoints of the blocks) is contained in the interior of exactly one region. The random walk may then be decomposed into excursions within the different regions. Let $T_0=0$ and let $T_1$ be the first time the random walk reaches the boundary of the region having $X_{T_0}=0$ in its interior. For $k\geq 2$, let $T_k$ be the first time after $T_{k-1}$ that the random walk reaches the boundary of the region containing $X_{T_{k-1}}$ in its interior. Let $N:= \max\{ k: T_{k-1} < 2n \}$ be the number of excursions within regions started before time $2n$. The following lemma allows us to bound $N$ from above. 

\begin{lem}\label{Nexc}
For any environment $\w$,
\[
 \lim_{n\ra\infty} \frac{(\ln n)^2}{n} \ln P_\w \left( \max_{k\leq 2n} |X_k| \leq n^\gamma \;,\;  N \geq \frac{5n}{(\ln n)^{3-2\d}}  \right) = -\infty. 
\]
\end{lem}
\begin{proof}
As in \cite{ppQSubexp}, let $N_{2n}^{\la}$ denote the total number of instances before time $2n$ that the random walk crosses a biased block from right to left. Lemma 4 in \cite{ppQSubexp} implies that for any environment $\w$,
\begin{equation}\label{backub}
 \limsup_{n\ra\infty} \frac{(\ln n)^2}{n} \ln P_\w \left( N_{2n}^\la \geq \frac{2n}{(\ln n)^{3-2\d}} \right) = -\infty. 
\end{equation}
Also, note that on the event $\{ \max_{k\leq n} |X_k| \leq n^\gamma \}$ (cf. (45) in \cite{ppQSubexp}),
\[
 0 \leq N \leq \frac{2 n^\gamma}{(\ln n)^{1-\d}} + 2 N_{2n}^\la. 
\]
Indeed, there are at most $2 n^\gamma/(\ln n)^{1-\d}$ regions intersecting $[-n^\gamma, n^\gamma]$ and each of these can be crossed once without any left crossings of biased blocks. Since each region begins and ends with a biased block, additional excursions within regions can only be accomplished by crossing a biased block from right to left and each such right-left crossing of a biased block can contribute to at most two more excursions within regions. Therefore, for $n$ sufficiently large, 
\[
 P_\w \left( \max_{k\leq 2n} |X_k| \leq n^\gamma \;,\;  N \geq \frac{5n}{(\ln n)^{3-2\d}}  \right) \leq P_\w \left( N_{2n}^\la \geq \frac{2n}{(\ln n)^{3-2\d}} \right). 
\]
Recalling \eqref{backub} finishes the proof of Lemma \ref{Nexc}. 
\end{proof}

To finish the proof of Proposition \ref{mnconfine}, we need to show that
\be\label{Nsmallub}
 \limsup_{n\ra\infty} \frac{(\ln n)^2}{n} \ln P_\w\left(  \max_{k\leq 2n} |X_k| \leq n^\gamma \;,\;  N < \frac{5n}{(\ln n)^{3-2\d}}  \right) \leq - \frac{| \pi \ln \a|^2}{4 \gamma^2 }, \qquad P-a.s. 
\ee
We next modify the hitting times $T_k$ to account for exiting the interval $[-n^\gamma, n^\gamma]$ as well as reaching the boundary of a region. Let $T^*_0 = 0$, and $T^*_1$ be the first time the random walk either exits the interval $[-n^\gamma, n^{\gamma}]$ or reaches the boundary of the region containing $X_{T^*_0}$ in its interior. Similarly, for $k\geq 2$ let $T^*_k$ be the first time after $T^*_{k-1}$ that the random walk either exits the interval $[-n^\gamma, n^\gamma]$ or reaches the boundary of the region containing $X_{T^*_{k-1}}$ in its interior. 
Note that on the event $\{ \max_{k\leq 2n} |X_k| \leq n^{\gamma} \}$ we have $T_k = T^*_k$ for all $k < N$. 
Then, for any $K\in\N$ and any $\l>0$, the strong Markov property implies that 
 \begin{align}
& P_\w\left(  \max_{k\leq 2n} |X_k| \leq n^\gamma \;,\;  N = K  \right) 
\leq P_\w\left( \sum_{k=1}^K (T_k^*-T_{k-1}^*) > 2n; \; | X_{T^*_{k}}|  \leq n^\gamma , \; \forall k < K \right) \nonumber \\ 
%&\leq P_\w\left( \sum_{k=1}^K (T^*_k-T^*_{k-1}) \geq n \right) \\ \nonumber
&\quad \leq e^{-2\l n} E_\w\left[ \exp\left\{ \l \sum_{k=1}^K (T^*_k-T^*_{k-1}) \right\} \indd{| X_{T^*_{k}}|  \leq n^\gamma , \; \forall k < K} \right] \nonumber \\
&\quad \leq e^{-2\l n} E_\w\left[ \exp\left\{ \l \sum_{k=1}^{K-1} (T^*_k-T^*_{k-1}) \right\} E_\w^{X_{T_{K-1}^*}} \!\!\left[ e^{ \l T^*_1 } \right] \indd{| X_{T^*_{k}}|  \leq n^\gamma , \; \forall k < K}\right]. \label{iterate}
\end{align}
The inner expectation in the last line above is of the form $E^x_\w\left[ e^{\l T_I} \right]$, where $T_I = \min\{ k\geq 0: X_k \notin I \}$ is the first time the random walk exits the interval $I$, 
$I = R_j^\circ \cap [-n^\gamma, n^\gamma]$, and $R_j$ is the region with $x$ in its interior $R_j^\circ$.
As was shown in \cite{ppzSubexp} (see equation (47) through Lemma 3) we can choose an appropriate $\l$ and bound such expectations uniformly in $x$, $I$, and all environments $\w$ with $\w_x \geq 1/2$ for all $x\in I$. Indeed, for any $\rho \in (0,1)$ there exists a constant $\chi(\rho) \in (1,\infty)$ such that
\[
 E^x_\w\left[ e^{\l' T_I} \right] \leq \chi(\rho), \quad\text{where}\quad \l' =  \frac{(1-\rho)\pi^2}{8(|I|-1)^2}.
\]
Note that $\l'$ decreases in the size of the interval $I$. 
Thus, choosing $\l$ according to the largest region intersecting $[-n^\gamma, n^\gamma]$ and iterating the computation in \eqref{iterate} we obtain that 
\begin{equation}\label{fixedKub}
 P_\w\left(  \max_{k\leq 2n} |X_k| \leq n^\gamma \;,\;  N = K  \right) \leq \exp\left\{ - n  \frac{(1-\rho)\pi^2}{4(\max_j |R_j^\circ \cap [-n^\gamma, n^\gamma]|)^2} \right\} \chi(\rho)^K. 
\end{equation}
It remains to give an upper bound on $\max_j |R_j^\circ \cap [-n^\gamma, n^\gamma]|$. 
To this end, for any $p\in(0,1)$ and $x\in[0,1]$ let $\L^*_p(x) = x \ln (x/p) + (1-x)\ln((1-x)/(1-p))$ be the large deviation rate function for a Bernoulli($p$) random variable. Then, 
Lemma 3 in \cite{ppQSubexp} implies that for any $\e$ and $\xi$ fixed as in the definition of the blocks, $P-a.s.$, there exists an $n_1(\w,\e,\xi)$ such that 
\[
 \max_j |R_j^\circ \cap [-n^\gamma, n^\gamma]| \leq \frac{(1+\e)\gamma}{\Lambda^*_{P(\w_0 \geq 1/2 + \xi)}(\e)} \ln n, \qquad \forall n\geq n_1(\w,\e,\xi). 
\]
Therefore, recalling \eqref{fixedKub},
\begin{equation}
 \limsup_{n\ra\infty} \frac{(\ln n)^2}{n} \ln P_\w\left(  \max_{k\leq 2 n} |X_k| \leq n^\gamma \;,\;  N < \frac{5 n}{(\ln n)^{3-2\d}}  \right) \leq - \frac{(1-\rho) \pi^2 \left(\L^*_{P(\w_0 \geq 1/2 + \xi)}(\e)\right)^2}{4(1+\e)^2 \gamma^2}.  \label{allKub}
\end{equation}
Note that the left hand side does not depend on $\rho$, $\e$ or $\xi$ and that 
\[
 \lim_{\xi\ra 0} \lim_{\e\ra 0} \Lambda^*_{P(\w_0 \geq 1/2 + \xi)}(\e) = \lim_{\xi \ra 0} - \ln P(\w_0 < 1/2 + \xi) = -\ln P(\w_0 = 1/2) = -\ln\a.
\]
Thus, taking $\rho$, $\e$, and then $\xi$ to zero in \eqref{allKub} implies \eqref{Nsmallub} and thus finishes the proof of Proposition \ref{mnconfine}. 
\end{proof}

\begin{proof}[Proof of Theorem \ref{margnestlingcase}.]
We first prove \eqref{mndispub} for $\b<2$ which gives an upper bound on the maximal displacement of a bridge. First, note that
\be \label{mnsub}
 P_\w\left( \max_{k\leq 2n} |X_k| \geq \frac{n}{(\ln n)^\b} ,\; X_{2n}= 0 \right) \leq P_\w\left( T_{-n/(\ln n)^\b} < n \right) + P_{\theta^{n/(\ln n)^\b} \w}\left( T_{-n/(\ln n)^\b} < n \right) .
\ee
The shift invariance of $P$ and Chebychev's inequality imply that for any $\d>0$
\begin{align}
& P\left( P_\w\left( \max_{k\leq 2n} |X_k| \geq \frac{n}{(\ln n)^\b} ,\; X_{2n}= 0 \right) \geq 2 e^{-\d n/(\ln n)^\b} \right) \nn \\
&\qquad \leq 2 P\left(  P_\w\left( T_{-n/(\ln n)^\b} < n \right) > e^{-\d n/(\ln n)^\b} \right) \nn \\
&\qquad \leq 2 e^{\d n/(\ln n)^\b} \P\left( T_{-n/(\ln n)^\b} < n \right). \label{annbacktrack}
\end{align}
Lemma 2.2 in \cite{dpzTE1D} gives that $\P( T_{-x} < \infty) \leq \frac{(E_P \rho_0)^x}{1- E_P \rho_0}$ for any $x\geq 0$ (note that Assumption \ref{MNasm} implies that $E_P \rho_0 < 1$). Therefore, if $0<\d< - \ln (E_P \rho_0)$, then \eqref{annbacktrack} and the Borel-Cantelli Lemma imply that $ P_\w\left( \max_{k\leq 2n} |X_k| \geq \frac{n}{(\ln n)^\b} ,\; X_{2n}= 0 \right) \leq 2 e^{-\d n/(\ln n)^\b}$ for all sufficiently large $n$, $P-a.s.$
Therefore, Lemma \ref{mnatzero} implies that, $P-a.s.$, for $\d< - \ln (E_P \rho_0)$, $C> \frac{|\pi \log \a|^2}{4}$, and $n$ large enough,
\[
 P_\w\left( \max_{k\leq 2n} |X_k| \geq \frac{n}{(\ln n)^\b} \; \Bigl| \; X_{2n} = 0 \right) \leq \frac{2 e^{-\d n/(\ln n)^\b}}{ e^{-C n/(\ln n)^2}}. 
\]
If $\b<2$ the right hand side vanishes as $n\ra\infty$. 

To get the lower bound on the maximal displacement fix $\gamma \in (0,1)$. Then,
\begin{align*}
  P_\w\left( \max_{k\leq 2n} |X_k| > n^{\gamma } \; \Bigl| \; X_{2n} = 0 \right) &= 1 - P_\w\left( \max_{k\leq 2n} |X_k| \leq n^\gamma \; \Bigl| \; X_{2n} = 0 \right) \\
&\geq 1 - \frac{P_\w\left( \max_{k\leq 2n} |X_k| \leq n^\gamma \right)}{P_\w\left( X_{2n} = 0 \right)}
\end{align*}
Lemma \ref{mnatzero} and Proposition \ref{mnconfine} imply that the right hand side tends to 1 as $n\ra\infty$, $P-a.s.$, thus completing the proof of \eqref{mndisplb}. 
\end{proof}

\section{Case III: Non-nestling environment}\label{nonnestlingsection}
In this section we will make the following assumptions on the environment. 
\begin{asm}\label{NNPMasm}
 $\w_{\min} > \frac{1}{2}$, and $\a:=P(\w_0 = \w_{\min} ) \in(0,1)$.
\end{asm}
\begin{asm}\label{NNGasm}
 There exists an $\eta>0$ such that $P(\w_0 \in (\w_{\min}, \w_{\min}+\eta)) = 0$. 
\end{asm}
Assumption \ref{NNPMasm} is the crucial assumption needed for the main results. Assumption \ref{NNGasm} is a technical assumption and all the main results should be true under only Assumption \ref{NNPMasm}. Note that Assumption \ref{NNPMasm} implies Assumption \ref{Tasm}.

The main result in this section is that the magnitude of displacement in bridges is the same in the non-nestling case as it is in the marginally nestling case. 
\begin{thm}\label{nonnestlingcase}
  Let Assumptions \ref{UEIIDasm}, \ref{NNPMasm}, and \ref{NNGasm} hold. Then
%, for $P-a.e.$ environment $\w$,
%\[
% \lim_{n\ra\infty} P_\w\left( \max_{k\leq 2n} |X_k| \leq \frac{n}{(\ln n)^\b} \; \Bigl| \; X_{2n} = 0 \right) 
%= \begin{cases}
%   0 & \b < 2 \\
%   1 & \b > 2.
%  \end{cases}
%\]
for any $\b < 2$,
\begin{equation} \label{nndispub}
 \lim_{n\ra\infty} P_\w\left( \max_{k\leq 2n} |X_k| \geq \frac{n}{(\ln n)^\b} \; \Bigl| \; X_{2n} = 0 \right) = 0, \quad P-a.s.,
\end{equation}
and for any $\gamma < 1$,
\begin{equation}  \label{nndisplb}
  \lim_{n\ra\infty} P_\w\left( \max_{k\leq 2n} |X_k| \geq n^\gamma \; \Bigl| \; X_{2n} = 0 \right) = 1, \quad P-a.s.
\end{equation}
\end{thm}

The quenched large deviation principle for $X_n/n$, established in  \cite{cgzLDP,gdhLDP} implies that there exists a deterministic function $I(v)$ such that $P_\w(X_n \approx v) \approx e^{-nI(v)}$. While there is a probabilistic formula for $I(v)$, it is not computable in practice for most values of $v$. However, when $v=0$, the value is known. If the environment is nestling or marginally nestling then $I(0)=0$, but in the non-nestling case 
\[
 I(0) = -\frac{1}{2}\ln\left( 4 \w_{\min}(1-\w_{\min}) \right) > 0.
\]
%The following limit exists by subadditivity,
%\[
% \lim_{n\ra\infty} \frac{1}{2n} \log P_\w( X_{2n} = 0 ) = - I(0).
%\]
As was the case for Theorem \ref{margnestlingcase}, the keys to the proof of Theorem \ref{nonnestlingcase} are the asymptotics of the quenched probabilities 
to be at the origin at time $2n$ or to stay confined to the interval $[-n^\gamma, n^\gamma]$ for the first $2n$ steps of the random walk. As will be shown below, both of these probabilities have the same exponential rate of decay:
\[
 \lim_{n\ra\infty} \frac{1}{n} \ln P_\w(X_{2n} = 0) = \lim_{n\ra\infty} \frac{1}{n} \ln P_\w\left( \max_{k\leq 2n} |X_k| \leq n^\gamma \;,\; X_{2n} = 0 \right) = -2I(0), \qquad P-a.s.
\]
Thus, in order to prove Theorem \ref{nonnestlingcase}, more precise asymptotics of the decays of these quenched probabilities will be needed. 

A key tool of our analysis in the non-nestling case will be a transformation of the environment $\w$ into an environment $\tw$ that is marginally nestling. 
For any environment $\w$, let $\tw$ be the environment defined by 
\begin{equation}\label{twdef}
 \tw_x = \frac{\rho_{\max}}{\rho_x + \rho_{\max}}, x\in \Z
\end{equation}
where $\rho_{\max}= \frac{1-\w_{\min}}{\w_{\min}}$. Note that under Assumption \ref{NNPMasm}, $P(\tw_x \geq \frac{1}{2}) = 1$ and $P(\tw_x = \frac{1}{2}) = P(\w_x = \w_{\min}) = \a>0$. 
%Therefore Lemma \ref{mnatzero} implies that
%\begin{equation}\label{tiltedatzero}
% \lim_{n\ra\infty} \frac{(\ln n)^2}{n} \ln P_{\tw}( X_{2n} = 0 ) = -\frac{ |\pi \ln \a|^2 }{4}, \quad P-a.s. 
%\end{equation}
The usefulness of this transformation stems from the following Lemma.
\begin{lem}\label{COMlem}
Let Assumptions \ref{NNPMasm} and \ref{NNGasm} hold, and let 
\[
 B_n := \# \{ k < 2n: \w_{X_k} > \w_{\min} \}.
\]
Then there exist constants $c_1,c_2\in(0,1)$ (depending on $\w_{\min}$ and $\eta$) such that for any event $A \in \s(X_0,X_1, \ldots, X_{2n})$ depending only on the first $2n$ steps of the random walk such that $A \subset \{X_{2n}=0\}$,
\[
 e^{-2I(0)n} E_{\tw}\left[ c_1^{B_n} \mathbf{1}_A \right] \leq P_\w( A ) \leq e^{-2I(0)n} E_{\tw}\left[ c_2^{B_n} \mathbf{1}_A \right], \quad P-a.s.
\]
\end{lem}
\begin{proof}
Let $X_{[0,2n]} = (X_0, X_1, \ldots X_{2n})$ denote the path of the random walk in the time interval $[0,2n]$.
Then, 
\begin{equation}\label{tilting}
 P_\w( A ) = E_{\tw}\left[ \frac{dP_\w}{dP_{\tw}}(X_{[0,2n]}) \mathbf{1}_A \right], 
\end{equation}
where for any $\mathbf{x}_{[0,2n]}= (x_0,x_1,\ldots, x_{2n})$ we have 
\[
 \frac{dP_\w}{dP_{\tw}}(\mathbf{x}_{[0,2n]}) = \frac{P_\w( X_k= x_k, \; \forall k\in[0,2n] )}{P_{\tw}( X_k= x_k, \; \forall k\in[0,2n] )} =\prod_{k=0}^{2n-1} \frac{P_\w^{x_k}( X_1 = x_{k+1} )}{P_{\tw}^{x_k}( X_1 = x_{k+1} )}.
\]
Recalling the formula for $\tw_x$ in \eqref{twdef}, we obtain that 
\[
 \frac{P_\w^{x_k}( X_1 = x_{k+1} )}{P_{\tw}^{x_k}( X_1 = x_{k+1} )}  =
\begin{cases}
 \w_x(\rho_x + \rho_{\max})\rho_{\max}^{-1} & \text{ if } x_{k+1}= x_k+1 \\
 \w_x(\rho_x + \rho_{\max}) & \text{ if } x_{k+1}= x_k-1.
\end{cases}
\]
Then since $X_{2n}=0$ implies that exactly half of the first $2n$ steps of the random walk are to the right we obtain that,
\begin{equation}\label{COMatzero}
 \frac{dP_\w}{dP_{\tw}}(X_{[0,2n]}) = \rho_{\max}^{-n} \prod_{k=0}^{2n-1}\w_{X_k}(\rho_{X_k} + \rho_{\max}), \quad \forall X_{[0,2n]} \in\{ X_{2n}=0\}. 
\end{equation}

%Next, note that since $\w_x(\rho_x + \rho_{\max})$ is decreasing in $\w_x$. 
%we have that 
%\[
%\w_x(\rho_x + \rho_{\max}) \leq 2(1-\w_{\min}), \quad \forall x\in\Z \quad P-a.s.
%\]
%Also, 
If $\w_x = \w_{\min}$ then $\w_x(\rho_x + \rho_{\max}) = 2(1-\w_{\min})$, and since $\w_x(\rho_x + \rho_{\max})$ is decreasing in $\w_x$ there exist constants $c_1,c_2 <1$ (depending on $\w_{\min}$ and $\eta$) such that 
\[
 c_1 2(1-\w_{\min}) \leq \w_x(\rho_x + \rho_{\max}) \leq c_2 2(1-\w_{\min}), \quad \forall \w_x \in [\w_{\min} + \eta, 1]. 
\]
Since Assumption \ref{NNGasm} implies that $P\left(\w_x \in \{\w_{\min}\} \cup [\w_{\min} + \eta, 1] \right) = 1$, 
\eqref{COMatzero} implies that, $P-a.s.$,
%\[
% \frac{dP_\w}{dP_{\tw}}(X_{[0,2n]}) \leq \rho_{\max}^{-n} (2 (1-\w_{\min}))^{2n} c^{B_n} = (4 \w_{\min} (1-\w_{\min}))^n c^{B_n}, \quad \forall X_{[0,2n]} \in\{ X_{2n}=0\}.
%\]
\begin{equation}\label{COMlbub}
 \rho_{\max}^{-n} (2 (1-\w_{\min}))^{2n} c_1^{B_n} \leq \frac{dP_\w}{dP_{\tw}}(X_{[0,2n]}) \leq \rho_{\max}^{-n} (2 (1-\w_{\min}))^{2n} c_2^{B_n}, \quad \forall X_{[0,2n]} \in\{ X_{2n}=0\}.
\end{equation}
Since $\rho_{\max}^{-n} (2 (1-\w_{\min}))^{2n} = (4 \w_{\min} (1-\w_{\min}))^n = e^{-2I(0)n}$, applying \eqref{COMlbub} to \eqref{tilting} completes the proof of the Lemma. 
\end{proof}

We now apply Lemma \ref{COMlem} to prove precise decay rates of certain quenched probabilities. 
\begin{prop}\label{nnconfine}
Let Assumptions \ref{UEIIDasm}, \ref{NNPMasm}, and \ref{NNGasm} hold. Then, for any $\gamma \in (0,1]$,
\[
 \limsup_{n\ra\infty} \frac{(\ln n)^2}{n} \left\{ \ln P_\w \left( \max_{k\leq 2n} |X_k| \leq n^\gamma \; , \; X_{2n} = 0 \right) + 2n I(0) \right\} = - \frac{| \pi \ln \a|^2}{\gamma^2 }, \qquad P-a.s. 
\]
\end{prop}
Before giving the proof of Lemma \ref{COMlem} we state the following corollary which is obtained by taking $\gamma = 1$. 
\begin{cor}\label{nnatzero}
Let Assumptions \ref{UEIIDasm}, \ref{NNPMasm}, and \ref{NNGasm} hold. Then,
\[
\lim_{n\ra\infty} \frac{(\ln n)^2}{n} \left\{ \ln P_\w(X_{2n} = 0) + 2n I(0) \right\} = -|\pi \ln \a|^2. 
\]
%\begin{align}
%-|\pi \ln \a|^2 &\leq \liminf_{n\ra\infty} \frac{(\ln n)^2}{n} \left( \ln P_\w(X_{2n} = 0) + 2n I(0) \right) \label{nestatzerolb}\\
%&\qquad \leq  \limsup_{n\ra\infty} \frac{(\ln n)^2}{n} \left( \ln P_\w(X_{2n} = 0) + 2n I(0) \right) \leq -\frac{|\pi \ln \a|^2}{4}. \label{nestatzeroub}
%\end{align}
\end{cor}

\begin{rem}
Since $B_n \geq 0$ by definition and the environment $\tw$ is marginally nestling, a simple application of Lemma \ref{COMlem} and Proposition \ref{mnconfine} implies that 
\begin{align*}
& \limsup_{n\ra\infty} \frac{(\ln n)^2}{n} \left\{ \ln P_\w \left( \max_{k\leq 2n} |X_k| \leq n^\gamma \; , \; X_{2n} = 0 \right) + 2n I(0) \right\} \\
&\qquad \leq \limsup_{n\ra\infty} \frac{(\ln n)^2}{n}  \ln P_{\tw} \left( \max_{k\leq 2n} |X_k| \leq n^\gamma \; , \; X_{2n} = 0 \right) 
\leq - \frac{| \pi \ln \a|^2}{4 \gamma^2 }, \qquad P-a.s. 
\end{align*}
This does not quite give the correct upper bound obtained in Proposition \ref{nnconfine} and reflects a subtle but important difference between the way that a RWRE is confined to the interval $[-n^\gamma, n^\gamma]$ in marginally nestling and non-nestling environments. 
The proof of Proposition \ref{mnconfine} suggests that under Assumption \ref{MNasm}, $B_n$ is typically of the order $n/\ln n$ on the event $\{ \max_{k\leq 2n} |X_k| \leq n^\gamma \}$. 
In contrast, the proof of Proposition \ref{nnconfine} below will suggest that under Assumptions \ref{NNPMasm} and \ref{NNGasm}, $B_n$ is typically less than $n/(\ln n)^{2-\d}$ for any $\d>0$ on the event $\{ \max_{k\leq 2n} |X_k| \leq n^\gamma \}$. 
%In marginally nestling environments, the seemingly best strategy is to confine the random walk to a long stretch of fair sites with an (implied) added reflection on the left edge. Since the length of the fair interval is  $\bigo(\ln n)$ and the walk is confined for $\bigo(n)$ steps, the number of visits to the left edge is $\bigo(n/(\ln n))$. Thus, typically there are at least $n/\ln n$ visits to non-fair sites in the marginally nestling case. 
%In contrast, the proof of Proposition \ref{nnconfine} below will show that when the environment is non-nestling, $B_n \leq n/(\ln n)^{2-\d}$ for any $\d>0$. 
\end{rem}

\begin{proof}[Proof of Proposition \ref{nnconfine}]
 We first give a lower bound. Lemma \ref{COMlem} implies that, $P-a.s.$,
\begin{align}
& P_\w\left( \max_{k\leq 2n} |X_k| \leq n^\gamma \; , \; X_{2n} = 0 \right) \nonumber \\
&\quad \geq P_\w\left( \max_{k\leq 2n} |X_k| \leq n^\gamma \; , \; X_{2n} = 0 \;,\; B_n \leq \frac{2 n}{(\ln n)^3} \right) \nonumber \\
&\quad \geq e^{-2I(0)n} c_1^{n/(\ln n)^3} P_{\tw}\left( \max_{k\leq 2n} |X_k| \leq n^\gamma \; , \; X_{2n} = 0 \;,\; B_n \leq \frac{2 n}{(\ln n)^3} \right), \label{nnconfinelb}
\end{align}
where $B_n= \#\{ k\leq 2n: \tw_{X_k} > 1/2 \}$ is the number of visits to biased sites in the environment $\tw$ before time $2n$. 
Since $\tw$ is a marginally nestling environment, we can obtain a lower bound for the last probability above in a similar manner as the lower bounds for Lemma \ref{mnatzero} and Proposition \ref{mnconfine}. 
However, because of the added condition that $B_n \leq \frac{2 n}{(\ln n)^3}$, instead of adding a reflection at the left edge of the longest fair stretch in $[0, n^\gamma]$ (or $[0, n/(\ln n)^3]$ when $\gamma = 1$), we instead force the random walk to stay strictly inside the longest fair stretch for time $2n$. Since the longest fair stretch in $[0, n^\gamma]$ (or $[0, n/(\ln n)^3]$ when $\gamma = 1$) in the environment $\tw$ is of size $(\gamma/|\ln \a| + o(1))\ln n$, we obtain from uniform ellipticity that, $P-a.s.$,
\begin{align}
& \liminf_{n\ra\infty} \frac{(\ln n)^2}{n} \ln P_{\tw}\left( \max_{k\leq 2n} |X_k| \leq n^\gamma \; , \; X_{2n} = 0 \;,\; B_n \leq \frac{2 n}{(\ln n)^3} \right) \nonumber \\
&\quad \geq \liminf_{n\ra\infty} \frac{(\ln n)^2}{n} \ln P_{1/2}\left( \max_{k\leq 2n} |X_k| \leq \frac{\gamma}{2 |\ln \a|}\ln n \right) = - \frac{|\pi \ln \a|^2}{\gamma^2}. \label{confineLFS}
\end{align}
Note that the last equality is again obtained from Lemma \ref{SRWsmalldev}. Combining \eqref{nnconfinelb} and \eqref{confineLFS} completes the proof of the needed lower bound.

To get a corresponding upper bound, first note that by Lemma \ref{COMlem}, $P-a.s.$, considering the events $\{B_n > \frac{n}{(\ln n)^{2-\d}}\}$ and $\{B_n \leq \frac{n}{(\ln n)^{2-\d}}\}$,
\begin{align*}
& P_\w\left( \max_{k\leq 2n} |X_k| \leq n^\gamma \; , \; X_{2n} = 0 \right) \\
&\qquad \leq e^{-2I(0)n} \left\{ c_2^{n/(\ln n)^{2-\d}} + P_{\tw}\left( \max_{k\leq 2n} |X_k| \leq n^\gamma \; , \; B_n \leq \frac{n}{(\ln n)^{2-\d}} \right) \right\}.
\end{align*}
Then, since $c_2 <1$, it is enough to show that for some $\d>0$
\[
 \limsup_{n\ra\infty} \frac{(\ln n)^2}{n} \ln P_{\tw}\left( \max_{k\leq 2n} |X_k| \leq n^\gamma \;,\; B_n \leq \frac{n}{(\ln n)^{2-\d}} \right) \leq - \frac{|\pi \ln \a|^2}{\gamma^2}. 
\]
%(Compare to Proposition \ref{mnconfine} to see that the second event involving $B_n$ matters!).\\
For any $\e>0$ and $n$ large enough, $B_n \leq n/(\ln n)^{2-\d} < \e n$ implies that the random walk must spend at least $(2-\e)n$ steps at sites that are fair in $\tw$. 
Theorem 3.2.1 in \cite{dzLDTA} implies that, $P-a.s.$, for all $n$ large enough the length of the longest fair stretch in $[-n^\gamma, n^\gamma]$ in the environment $\tw$ is less than $\frac{(1+\e)\gamma}{|\ln \a|} \ln n$. 
Therefore, $P-a.s.$, for $n$ sufficently large,
\[
P_{\tw}\left( \max_{k\leq 2n} |X_k| \leq n^\gamma \;,\; B_n \leq \frac{n}{(\ln n)^{2-\d}} \right) 
\leq P_{\tw}(\tau \geq \e n)
+ P_{1/2}\left( \sum_{j\leq n/(\ln n)^{2-\d}} \s_{j,n} \geq 2n(1-\e) \right),
\]
where $\tau:= \min \{ k\geq 0: \tw_{X_k} > 1/2 \}$ is the first time the random walk reaches a biased site in $\tw$ and the $\s_{j,n}$ are i.i.d.\ with common distribution equal to that of the first time a simple random walk started at $x=1$ exits the interval $[1,(1+\e)\gamma \ln n/|\ln \a|]$. 
For any fixed environment $\tw$, $\tau$ is the exit time of a simple symmetric random walk from a fixed bounded interval and thus $P_{\tw}(\tau > \e n) $ decays exponentially fast. Therefore, the proof of the upper bound is reduced to showing that for some $\d>0$,
\be\label{subexpterm2}
 \limsup_{n\ra\infty} \frac{(\ln n)^2}{n} \ln P_{1/2}\left( \sum_{j\leq n/(\ln n)^{2-\d}} \s_{j,n} \geq 2n(1-\e) \right) \leq - \frac{|\pi \ln \a|^2}{\gamma^2}. 
\ee
%Therefore,
%\be \label{taudecay}
%\lim_{n\ra\infty} \frac{(\ln n)^2}{n} \ln P_{\tw}( \tau > \e n) = - \infty, \qquad P-a.s.
%\ee
%Thus, we only need to bound the second term on the right in \eqref{fairexc}. 
To this end, note that for any $\l>0$,
\be\label{chebychev}
 P_{1/2}\left( \sum_{j\leq n/(\ln n)^{2-\d}} \s_{j,n} \geq (1-\e)2n \right) \leq e^{-\l (1-\e) 2 n} \left\{ E_{1/2}\left[ e^{\l \s_{1,n}} \right] \right\}^{n/(\ln n)^{2-\d}}. 
\ee
To complete the proof of the upper bound, we need the following Lemma.
% which allows us to bound $E_{1/2}[ e^{\l \s_1} ]$ for a specific choice of $\l$. 
\begin{lem}\label{MGFub}
Let $2\ell$ be an integer greater than $1$ and let $\e \in(0,1)$. 
Then, there exists a constant $C_1<\infty$ depending only on $\e$ such that for $\l(\e,\ell) := \frac{(1-\e)^2 \pi^2}{8\ell^2}$,
\[
 E_{1/2}\left[ e^{\l(\e,\ell) \s} \right] < 1 + \frac{C_1}{\ell},
\]
where $\s$ is the first time a simple random walk started at $x=1$ exits the interval $[1, 2\ell -1]$.
\end{lem}
Using Lemma \ref{MGFub}, we complete the proof of the upper bound in Proposition \ref{nnconfine}. 
Let $\l = \l(\e,\ell(n)) = \frac{(1-\e)^2 \pi^2}{8\ell(n)^2}$ in \eqref{chebychev}, where $2\ell(n)-1 = \fl{(1+\e)\gamma \ln n/|\ln \a|}$. Then Lemma \ref{MGFub} implies that for any $\d \in (0,1)$,
\begin{align*}
 & \limsup_{n\ra\infty} \frac{(\ln n)^2}{n} \ln P_{1/2}\left( \sum_{j\leq n/(\ln n)^{2-\d}} \s_{j,n} \geq (1-\e)2n \right)  \\
& \qquad \leq \limsup_{n\ra\infty} - \frac{(1-\e)^3 \pi^2 (\ln n)^2}{4 \ell(n)^2} +  (\ln n)^\d \ln \left( 1+ \frac{C_1}{\ell(n)} \right) \\
& \qquad = - \frac{(1-\e)^3 |\pi \ln \a|^2}{(1+\e)^2 \gamma^2} .
\end{align*}
Taking $\e\ra 0$ finishes the proof of \eqref{subexpterm2}. 
It remains only to give the proof of Lemma \ref{MGFub}. 
\begin{proof}[Proof of Lemma \ref{MGFub}]
As was shown in the proof of Lemma 4 in \cite{gNodes}, the moment generating function for the first exit time of a simple random walk from an interval may be solved explicitly. In particular, for $0 \leq \lambda < \lambda_{\rm crit} (\ell) = -\ln\left(\cos\frac{\pi}{2\ell}\right)$ we obtain that 
\[
 E_{1/2}\left[ e^{\l \s} \right] = \frac{\cos (c_\lambda(\ell -1))}{\cos(c_\lambda\ell)} = e^{-\l} + \sqrt{1-e^{-2\l}} \tan( c_\l \ell),
\]
where $c_\lambda = \arccos(e^{-\lambda})$.
In particular, we get the simple upper bound 
\be\label{MGFsub}
 E_{1/2}\left[ e^{\l \s} \right] < 1 + \sqrt{2 \l} \tan( c_\l \ell).
\ee
Recalling the definitions of $\l(\e,\ell)$ and $c_\l$, and using the fact that $\arccos(e^{-x}) < \sqrt{2x}$ for any $x>0$ we obtain that 
\[
 c_{\l(\e,\ell)}\ell = \arccos\left(\exp\left\{ - \frac{(1-\e)^2\pi^2}{8\ell^2} \right\} \right) \ell < \frac{(1-\e)\pi}{2}. 
\]
Therefore, $\tan(c_{\l(\e,\ell)} \ell)$ is bounded above, uniformly in $\ell$. 
Recalling \eqref{MGFsub} we obtain
\[
 E_{1/2}\left[ e^{\l(\e,\ell) \s} \right] < 1 + \sqrt{2 \l(\e,\ell)} \tan\left( \frac{(1-\e)\pi}{2} \right) = 1 + \frac{(1-\e)\pi}{2\ell} \tan\left( \frac{(1-\e)\pi}{2} \right). 
\]
This completes the proof of Lemma \ref{MGFub} with $C_1 = \frac{(1-\e)\pi}{2} \tan\left( \frac{(1-\e)\pi}{2} \right)$. 
\end{proof}
\end{proof}

We now are ready to give the proof of the main result in this section.

\begin{proof}[Proof of Theorem \ref{nonnestlingcase}]
Let $\tw$ the marginally nestling environment obtained from the non-nestling environment $\w$ as defined in \eqref{twdef}. 
Lemma \ref{COMlem} implies that 
%Recalling \eqref{tiltingcostub}, we obtain that 
\begin{align*}
 P_\w\left( \max_{k\leq 2n} | X_k | \geq \frac{n}{(\ln n)^\b} \; , \; X_{2n} = 0 \right)
&= E_{\tw} \left[ \frac{dP_{\w}}{dP_{\tw}}(X_{[0,2n]}) \indd{\max_{k\leq 2n} | X_k | \geq \frac{n}{(\ln n)^\b} \; , \; X_{2n} = 0} \right] \\
&\leq e^{-2n I(0)} P_{\tw}\left( \max_{k\leq 2n} | X_k | \geq \frac{n}{(\ln n)^\b} \; , \; X_{2n} = 0 \right). 
\end{align*}
Then, since $\tw$ is a marginally nestling environment, the proof of Theorem \ref{margnestlingcase} implies that for $\d$ sufficiently small, $P-a.s.$, for all $n$ sufficiently large,
\[
 P_\w\left( \max_{k\leq 2n} | X_k | \geq \frac{n}{(\ln n)^\b} \; , \; X_{2n} = 0 \right) \leq e^{-2n I(0) - \d \frac{n}{(\ln n)^\b} }. 
\]
This combined with Corollary \ref{nnatzero} completes the proof of \eqref{nndispub}. 
\eqref{nndisplb} follows directly from Proposition \ref{nnconfine} and Corollary \ref{nnatzero}. 
\end{proof}

\section{Further Refinements}\label{openproblemsection}

We believe that in the marginally nestling and non-nestling regimes, the maximal displacement of bridges is really of the order $n/(\ln n)^2$. 
%Here are some conjectures:
The proofs of Theorems \ref{margnestlingcase} and \ref{nonnestlingcase} suggest that the most likely way for the random walk to be back at the origin after $2n$ steps is to go quickly to a long interval $I$ with $\w_x = \w_{\min}$ for all $x\in I$ and then stay in the interval $I$ for almost $2n$ steps before returning quickly to the origin. 
%A larger maximal displacement $\max_{k\leq 2n} |X_k|$ obtains the benefit of finding a longer interval $I$ which is easier to stay in but has the cost of having to backtrack a farther distance. 
%Theorems \ref{margnestlingcase} and \ref{nonnestlingcase} are obtained by balancing the benefit and cost of such a large displacement. 
However, the longest such interval $I\subset [-n/(\ln n)^\beta, n/(\ln n)^\beta]$ has length of order $\ln n/|\ln \a|$ for any $\beta > 0$. Thus, it is difficult to show that the maximal displacement is at least $n/(\ln n)^\beta$ for any $\beta > 2$ when conditioned on $\{X_{2n}=0\}$. Nevertheless, for any fixed $\beta > 2$ the longest interval $I \subset [-n/(\ln n)^2, n/(\ln n)^2]$ with $\w_x = \w_{\min}$ for all $x\in I$ is with high probability not contained in $[-n/(\ln n)^\beta, n/(\ln n)^\beta]$. This leads us to the following conjecture. 

%\begin{con}\label{truly2}
% Let Assumption \ref{UEIIDasm} hold, and let $\w_{\min} \geq 1/2$ and $P(\w_0 = \w_{\min}) = \a>0$. Then, for any $\b>2$,
%\begin{align*}
% \liminf_{n\ra\infty} P_\w \left( \max_{k\leq 2n} |X_k| \geq \frac{n}{(\ln n)^\b}  \; \Bigl| \; X_{2n} = 0 \right) = 0, \quad P-a.s.,
%\end{align*}
%and
%\begin{align*}
% \limsup_{n\ra\infty} P_\w \left( \max_{k\leq 2n} |X_k| \geq \frac{n}{(\ln n)^\b}  \; \Bigl| \; X_{2n} = 0 \right) = 1, \quad P-a.s.,
%\end{align*}
%\end{con}
%\begin{rem}
%% The $\limsup$ should be easier I think. Look for environments where the longest fair (or weak-biased) stretch in 
%% $[0,n/(\ln n)^\b)$ is much shorter (relatively) than the longest stretch in $[0,n/(\ln n)^2)$. 
%Proving this conjecture will require very detailed asymptotics of the probability of a simple random walk to stay confined in an interval.
%\end{rem}
%
%Another possible refinement would be
\begin{con}\label{probconv}
 Let Assumption \ref{UEIIDasm} hold, and let $\w_{\min} \geq 1/2$ and $P(\w_0 = \w_{\min}) = \a>0$. Then, for any $\b>2$,
\begin{align*}
 \lim_{n\ra\infty} P_\w \left( \max_{k\leq 2n} |X_k| \geq \frac{n}{(\ln n)^\b}  \; \Bigl| \; X_{2n} = 0 \right) = 1, \quad \text{in $P$-probability.}
\end{align*}
\end{con}
%%%%%%%%%%%%%%%%%%%%%%%%%%%  BIBLIOGRAPHY  %%%%%%%%%%%%%%%%%%%%%%%
\bibliographystyle{alpha}
\bibliography{Bridge}

\end{document}